\documentclass[10pt]{amsart}

\usepackage{amsfonts,amssymb,amscd,amstext}
\usepackage[paper size={210mm,297mm},left=40mm,right=40mm,top=50mm,bottom=50mm]{geometry}
\usepackage[charter]{mathdesign}
\usepackage[colorlinks=true,linkcolor=blue,citecolor=red]{hyperref}
\usepackage[utf8]{inputenc}

\renewcommand{\le}{\leqslant}
\renewcommand{\ge}{\geqslant}
\newcommand{\ptl}{\partial}
\newcommand{\rr}{{\mathbb{R}}}
\newcommand{\la}{\lambda}
\newcommand{\hh}{H}
\newcommand{\h}{H}
\newcommand{\esf}{\mathbb{S}}

\newcommand{\nn}{\mathbb{N}}
\newcommand{\pp}{P}

\newcommand{\Sg}{\Sigma}
\newcommand{\sg}{\sigma}
\newcommand{\Om}{\Omega}
\newcommand{\eps}{\varepsilon}

\newcommand{\de}{\delta}

\newcommand{\vol}[1]{|#1|}

\newcommand{\cl}[1]{\text{\rm cl}(#1)}
\newcommand{\clb}{\mkern2mu\overline{\mkern-2mu B\mkern-2mu}\mkern2mu}

\DeclareMathOperator{\divv}{div}

\DeclareMathOperator{\intt}{int}
\DeclareMathOperator{\Lip}{Lip}

\DeclareMathOperator{\diam}{diam}
\DeclareMathOperator{\Ric}{Ric}
\DeclareMathOperator{\II}{II}

\newtheorem{theorem}{Theorem}[section]
\newtheorem{proposition}[theorem]{Proposition}
\newtheorem{lemma}[theorem]{Lemma}
\newtheorem{corollary}[theorem]{Corollary}

\theoremstyle{definition}

\newtheorem{remark}[theorem]{Remark}

\theoremstyle{remark}

\newenvironment{enum}{\begin{enumerate}
}{\end{enumerate}}

\numberwithin{equation}{section}

\begin{document}

\title{Isoperimetric inequalities in conically bounded convex bodies}

\author[M.~Ritor\'e]{Manuel Ritor\'e} \address{Departamento de
Geometr\'{\i}a y Topolog\'{\i}a \\
Universidad de Granada \\ E--18071 Granada \\ Espa\~na}
\email{ritore@ugr.es}
\author[S.~Vernadakis]{Efstratios Vernadakis} \address{Departamento de
Geometr\'{\i}a y Topolog\'{\i}a \\
Universidad de Granada \\ E--18071 Granada \\ Espa\~na}
\email{stratos@ugr.es}

\date{\today}

\thanks{Both authors have been supported by MICINN-FEDER grant MTM2010-21206-C02-01, and Junta de Andaluc\'{\i}a grants FQM-325 and P09-FQM-5088}

\begin{abstract}
We consider the problem of minimizing the relative perimeter under a volume constraint in the interior of a conically bounded convex set, i.e., an unbounded convex body admitting an \emph{exterior} asymptotic cone. Results concerning existence of isoperimetric regions, the behavior of the isoperimetric profile for large volumes, and a characterization of isoperimetric regions of large volume in conically bounded convex sets of revolution is obtained.
\end{abstract}


\maketitle

\thispagestyle{empty}

\bibliographystyle{abbrv}


\section{Introduction}
Throughout this paper we shall denote by $C\subset\rr^{n+1}$ an unbounded closed convex set with non-empty interior. We shall call such a set an \emph{unbounded convex body}. We are interested in the isoperimetric problem of minimizing the relative perimeter in the interior of $C$ under a volume constraint, specially for large volumes. The \emph{isoperimetric profile} of $C$ is the function $I_C:(0,+\infty)\to\rr^+$ assigning to any $v>0$ the infimum of the perimeter of sets of volume $v$. An \emph{isoperimetric region} $E\subset C$ is one whose perimeter equals $I_C(|E|)$, where $|E|$ is the volume of $E$. This implies $P(F)\ge P(E)$ for any $F\subset C$ such that $|F|=|E|$.

Given an unbounded convex body $C$, a classical notion in the theory of convex sets~is  that of the \emph{asymptotic cone} of $C$, or tangent cone at infinity, defined by $C_\infty=\bigcap_{\la>0}\la C$. We shall say that $C_\infty$ is \emph{non-degenerate} when $\dim C_\infty=\dim C=n+1$.  Assuming $C$ has a non-degenerate asymptotic cone, we can extract useful information on the isoperimetric profile $I_C$ of $C$ but, unfortunately, we need a stronger control on the large scale geometry of $C$ to get a more precise information on the geometry of large isoperimetric regions  in $C$. Thus we are led to consider \emph{conically bounded convex sets}. We shall say that a convex set $C$ is conically bounded if there exists a non-degenerate cone $C^\infty$ containing $C$, the \emph{exterior asymtotic cone} of $C$, so that the Hausdorff distance of $C_t=C\cap\{x_{n+1}=t\}$ and $(C^\infty)_t$ goes to zero when $t$ goes to infinity. When $C$ is conically bounded, $C^\infty$ coincides with $C_\infty$ up to translation. There are examples of convex sets $C$ with non-degenerate asymptotic cone that are not conically bounded.

Previous results on the isoperimetric profile of \emph{cylindrically bounded} convex bodies have been obtained by the authors in \cite{rv3}. For \emph{bounded} convex bodies we refer to \cite{rv1} and the references there. In convex cones, this isoperimetric problem has been considered by Lions and Pacella \cite{lions-pacella}, Ritoré and Rosales \cite{r-r} and Figalli and Indrei \cite{FI}. Outside convex bodies, possibly unbounded, isoperimetric inequalities have been established by Choe and Ritoré \cite{MR2338131}, and Choe, Ghomi and Ritoré \cite{MR2215458}, \cite{MR2329803}.

We have organized this paper into several sections. In Section~\ref{sec:preliminaries}, we fix the notation we shall use and give the appropriate background. In particular, we discuss the relation between conically bounded convex sets and unbounded convex bodies with non-degenerate asymptotic cone in Lemma~\ref{lem:C_infty=C^infty} and Remark~\ref{rem:C^_infty}. We also give the necessary background on finite perimeter sets.

In Section~\ref{sec:nondegenerate}, we consider convex bodies $C$ with non-degenerate asymptotic cone $C_\infty$ and we prove in Theorem~\ref{thm:optimaconi} that the isoperimetric profile $I_C$ of $C$ is always bounded from below by the isoperimetric profile of $I_{C_\infty}$, and that $I_C$ and $I_{C_\infty}$ are asymptotic. Inequality $I_C\ge I_{C_\infty}$ is interesting since it implies that the isoperimetric inequality of the convex cone $C_\infty$ also holds in $C$, although it is not sharp in general. We also show the continuity of the isoperimetric profile of $C$ in Lemma~\ref{lem:contprofconi}. 

In Section~\ref{sec:conically}, we consider conically bounded convex bodies with smooth boundary. The boundary of its exterior asymptotic cone out of the vertex is not regular in general as it follows from the discussion at the beginning of Section~\ref{sec:conically}. Assuming the regularity of this convex cone, we prove existence of isoperimetric regions for all volumes in Proposition~\ref{prp: exist isop reg conical}, and the concavity of the isoperimetric profile $I_C$ and of its power $I_C^{(n+1)/n}$ in Proposition~\ref{prp: I C is concave conical}. It is well-known \cite{MR2008339} that the concavity of $I_C^{(n+1)/n}$ implies the connectedness of isoperimetric regions in $C$. In a similar way to \cite{rv1} we prove a ``clearing-out'' result in Proposition~\ref{prp:leon rigot lem 42 coni}, and a lower density bound in Corollary~\ref{cor:lwrdnbnd}, that allow us to show in Theorem~\ref{thm:convrescnoi} a key convergence result: if we have a sequence isoperimetric regions in $C$ whose volumes go to infinity, then scaled them down to have constant volume, we have convergence of the scaled isoperimetric regions in \emph{Hausdorff distance} to a ball in the exterior asymptotic cone. Moreover, the boundaries of the scaled isoperimetric regions also converge in Hausdorff distance to the spherical cap that bounds this ball. This convergence can be improved to higher order convergence using Allard type estimates for varifolds using the estimate in Lemma~\ref{lem:boundmeancurv}.

In Section~\ref{sec:revolution}, we consider conically bounded sets of revolution. These sets are foliated, out of a compact set, by a family of spherical caps whose mean curvatures go to $0$ by Lemma~\ref{lem:foliation}. Using the results in the previous Section and an argument based on the Implicit Function Theorem, we show in Theorem~\ref{thm:foliation} that large isoperimetric regions are spherical caps meeting the boundary of the unbounded convex body in an orthogonal way.

The authors would like to thank the referee for the useful comments.

\section{Preliminaries}
\label{sec:preliminaries}


\subsection{Convex sets}
An \emph{unbounded convex body} $C\subset\rr^{n+1}$ will be a closed unbounded convex set with non-empty interior. A \emph{convex body} $C\subset\rr^{n+1}$ is a \emph{compact} convex set with non-empty interior. The dimension of a convex set $C\subset\rr^{n+1}$ is the dimension of the smallest affine subspace of $\rr^{n+1}$ containing $C$ and will be denoted by $\dim C$. We refer the reader to Schneider's monograph \cite{sch} for background on convex sets and functions.

Given $x\in C$ and $r>0$, we define the \emph{intrinsic ball} $B_C(x,r)=C\cap B(x,r)$, and  the corresponding closed ball $\clb_C(x,r)=C\cap \clb(x,r)$. For $E \subset C$, the \emph{relative boundary} of $E$ in the interior of $C$ is $\ptl_{C}E=\ptl E\cap \intt {C}$.

We say that a sequence of closed sets $\{E_i\}_{i\in\nn}\subset\rr^{n+1}$ converges in \emph{pointed Hausdorff distance} to some closed set $E$ if there exist a point $p\in\rr^{n+1}$ so that $\{E_i\cap \clb(p,r)\}_{i\in\nn}$ converges in Hausdorff distance to $E\cap \clb(p,r)$ for all $r>0$. This property is almost independent of the point $p$. If $q\in\rr^{n+1}$ and $E_i\cap\clb(q,r)_{i\in\nn}$ is non-empty for large $i$ then, applying the Kuratwoski criterion \cite[Thm.~1.8.7]{sch}, one easily sees that $E_i\cap \clb(q,r)$ converges to $\clb(q,r)$ in Hausdorff distance.

We define the \emph{asymptotic cone} $C_{\infty}$ of an unbounded convex body $C$ containing $0$ by
\begin{equation}
\label{eq:cinfty}
C_{\infty} = \bigcap_{\lambda > 0} \lambda C,
\end{equation}
where $\lambda C = \{\lambda x : x\in C \} $ is the image of $C$ under the homothety of center $0$ and ratio $\la$. If $p\in C$ and $h_{p,\la}$ is the homothety of center $p$ and ratio $\la$ then $\bigcap_{\la>0} h_{p,\la}(C)=p+C_\infty$ is a translation of $C_\infty$. Hence the shape of the asymptotic cone is independent of the chosen origin. When $C$ is bounded the set $C_{\infty}$ defined by \eqref{eq:cinfty} is a point.
It is known that ${\la}C$ converges, in the pointed Hausdorff topology, to the asymptotic cone $C_{\infty}$ \cite{bbi} and hence it satisfies $\dim C_\infty\le\dim C$. We shall say that the asymptotic cone is \emph{non-degenerate} if $\dim C_{\infty}=\dim C$. The solid paraboloid $\{z\ge x^2+y^2\}$ and the cilindrically bounded convex set $\{z\ge (1-x^2-y^2)^{-1}: x^2+y^2<1\}$ are examples of unbounded convex bodies with the same degenerate asymptotic cone $C_\infty=\{(0,0,z):z\ge 0\}$.

We define the \emph{tangent cone} $C_{p}$ of a (possibly unbounded) convex body $C$ at a given boundary point $p\in\ptl C$ as the closure of the set
\[
\bigcup_{\la > 0} h_{p,\la} (C).
\]
Tangent cones of convex bodies have been widely considered in convex geometry under the name of supporting cones \cite[\S~2.2]{sch} or projection cones \cite{MR920366}. From the definition it follows that $C_p$ is the smallest cone, with vertex $p$, that includes $C$. 


Let $K\subset \rr^{n+1}$ be a closed convex cone with interior points. It is known that the geodesic balls centred at the vertex are isoperimetric regions in $K$, \cite{lions-pacella}, \cite{r-r}, and that they are the only ones \cite{FI} for general convex cones, without any regularity assumption on the boundary. The invariance of $K$ by dilations centered at some vertex yields
\begin{equation}
\label{eq:isopkon}
I_K(v)=I_K(1)\,v^{n/(n+1)}.
\end{equation}


Given a convex body $C\subset\rr^n$ containing $0$ in its interior, its \emph{radial function} $\rho(C,\cdot):\esf^n\to\rr$ is defined by
\[
\rho(C,u)=\max\{\la\ge 0:\la u\in C\}.
\]
It easily follows that $\rho(C,u)\,u\in \ptl C$ for all $u\in\esf^n$.

Let $C\subset \rr^{n+1}$ be an unbounded convex body that can be written as the epigraph of a non-negative convex function over the hyperplane $x_{n+1}=0$. We shall say that $C$ is a \emph{conically bounded convex body} if, for every $t\ge 0$, the set $C_t=C\cap\{x_{n+1}=t\}$ is a convex body in the hyperplane $\{x_{n+1}=t\}$, and there exists a non-degenerate convex cone $C^{\infty}$ including $C$ such that
\begin{equation}
\label{eq:def ex asymp con}
\lim_{t\to\infty} \max_{|u|=1} |\rho(C_t,u)-\rho((C_{\infty})_t,u)|=0.
\end{equation}
We shall call $C^{\infty}$ the \emph{exterior asymptotic cone} of $C$. Because of our assumption of compactness of the slices $C_t$, the exterior asymptotic cone has a unique vertex. We have the following

\begin{lemma}
\label{lem:C_infty=C^infty}
Let $C\subset\rr^{n+1}$ be a conically bounded convex body. Then $C_{\infty}$ and $C^{\infty}$ coincide up to translation.
\end{lemma}

\begin{proof}
Assume $C$ is the epigraph of the convex function $f:\rr^n\to\rr^+$, and let $C^\infty$ be defined as the epigraph of the convex function $f^\infty:\rr^n\to\rr^+$. Since $C^\infty$ is a cone, assuming the origin is a vertex, we have $\la f^\infty(x)=f^\infty(\la x)$ for any $\la>0$ and $x\in\rr^n$.

Let us compute now the asymptotic cone $C_\infty$. From \eqref{eq:isopkon}, the point $(x,y)\in\rr^n\times\rr$ belongs to $C_\infty$ if and only if $(\mu x,\mu y)\in C$ for all $\mu>0$. This is equivalent to $y\ge \mu^{-1}\,f(\mu x)$ for all $\mu>0$. The family $\{f_\mu\}_{\mu>0}$, where $f_\mu$ is defined by $f_\mu(x)=\mu^{-1}\,f(\mu x)$, is composed of convex functions. The convexity of $f$ and the fact that $f(0)=0$ imply that $f_\mu(x)\le f_\beta(x)$ when $\mu\le\beta$. Hence the asymptotic cone of $C$ is the epigraph of the convex function $f_\infty=\sup_{\mu>0} f_\mu=\lim_{\mu\to\infty} f_\mu$. Observe that $\la f_\infty(x)=f_\infty(\la x)$ for all $\la>0$ and $x\in\rr^n$. Since $C\subset C^\infty$ we have $f\ge f^\infty$ and so
\[
f_\infty(x)\ge f_\mu(x)=\mu^{-1} f(\mu x)\ge \mu^{-1}f^\infty(\mu x)=f^\infty(x).
\]

Let us check now that $f_\infty=f^\infty$. Fix some $x\in\rr^n\setminus\{0\}$ and let $u=x/|x|$. Then $(x,f(x))\in\ptl C_{f(x)}$ and $\rho(C_{f(x)},u)=|x|$. If $\mu=f(x)/f^\infty(x)$ then $f^\infty(\mu x)=\mu f^\infty(x)=f(x)$. Hence $(\mu x,f^\infty(\mu x))$ belongs to $\ptl (C^\infty)_{f(x)}$, and $\rho((C^\infty)_{f(x)},u)$ is given by $\mu\,|x|=(f(x)/f_\infty(x))\, |x|$.
Hence we have
\[
|\rho(C_{f(x)},u)-\rho((C^\infty)_{f(x)},u)|=\bigg(\frac{f(x)}{f_\infty(x)}-1\bigg)\,|x|.
\]
Replacing $x$ by $\la x$ we get
\[
|\rho(C_{f(\la x)},u)-\rho((C^\infty)_{f(\la x)},u)|=\bigg(\frac{f(\la x)}{f_\infty(\la x)}-1\bigg)\,\la |x|.
\]
Letting $\la\to\infty$, we know that $f(\la x)$ converges to $\infty$ since $f(\la x)\ge \la f^\infty(x)$. By \eqref{eq:def ex asymp con} we obtain
\[
1=\lim_{\la\to +\infty}\frac{f(\la x)}{f^\infty(\la x)}=\lim_{\la\to +\infty}\frac{\la^{-1}f(\la x)}{\la^{-1}f^\infty(\la x)}=\frac{f_\infty(x)}{f^\infty(x)}.
\]
\end{proof}


\begin{remark}
\label{rem:C^_infty}
It is not difficult to produce examples of unbounded convex body with non-degenerate asymptotic cone which are not conically bounded. Simply consider the epigraph in $\rr^2$ of the convex function $f(x)=e^x-1$. Its asymptotic cone is the quadrant $x\le 0, y\ge 0$. On the other hand, there are no asymptotic lines to the graph of $f(x)$ when $x\to +\infty$.

Starting from this example we can produce higher dimensional ones: consider the reflection of $\{(x,f(x)): x\ge 0	\}$ with respect to the normal line $x+y=0$ to the graph~of $f(x)$ at $(0,0)$. This convex function can be used to produce higher dimensional unbounded convex bodies of revolution with non-degenerate asymptotic cone which are not conically bounded.
\end{remark}


\subsection{Sets of finite perimeter and isoperimetric regions}
The main references here are Giusti \cite{gi} and Maggi \cite{MR2976521}. Given $E \subset C$, we define the \emph{relative perimeter} of $ E$ in $\intt(C)$, by
\[
\pp_C(E) = \sup \Big \{ \int_E\divv \xi\, d{\h}^{n+1}, \xi \in \Gamma_0(C) ,\, |\xi| \le 1 \Big \},
\]
where $\Gamma_0(C)$ is the set of smooth vector fields with compact support in $\intt (C)$. We shall say that $E$ has \emph{finite perimeter} in $C$ if $\pp_C(E)<\infty$.

The \emph{volume} of $E$ is defined as the $(n+1)$-dimensional Hausdorff measure of $E$ and will be denoted by $|E|$. The $r$-dimensional Hausdorff measure will be denoted by $H^r$.

If $C,C'\subset\rr^{n+1}$ are convex bodies (possible unbounded) and $f:C\to C'$ is a Lipschitz map, then, for every $s>0$ and $E\subset C$, from the definition of Hausdorff measure, we get $\h^s(f(E)) \le \Lip(f)^s\, \h^s(E)$. This implies

\begin{lemma}
\label{lem:bilip}
Let $C,C'\subset\rr^{n+1}$ and $f:C\to C'$ a bilipschitz map then we have
\begin{equation}
\label{eq:bilip}
\begin{split}
\Lip(f ^{-1})^{-n}\,\pp_C(E) \le \pp_{f(C)}(f(E)) \le \Lip(f)^n\, \pp_C(E),
\\
\Lip(f ^{-1})^{-(n+1)}\,\vol{E} \le \vol{f(E)} \le \Lip(f)^{n+1}\,\vol{E}.
\end{split}
\end{equation}
\end{lemma}

\begin{proof}
The first line of inequalities holds when the boundary of $E$ is smooth. For general $E$ it follows by approximation by finite perimeter sets with smooth boundary. The second line is well-known.
\end{proof}

\begin{remark}
\label{rem:lipcomp}
 Let $M_ i$, $i=1,2,3$ be metric spaces and $f_i: M_i\to M_{i+1}$, $i=1,2$ be lipschitz maps, then $\Lip(f_2\circ f_1)\le \Lip(f_1)\Lip(f_2)$. Consequently if $g: M_1\to M_2$ is a bilipschitz map, then $1\le \Lip(g)\Lip(g ^{-1})$.
\end{remark}

\begin{remark}
\label{rem:lipschitz}
If $f:C_1\to C_2$ is a bilipschitz map between subsets of $\rr^{n+1}$, then $g:\la C_1\to \la C_2$, defined by $g(x)=\la f(\frac{x}{\la})$,  is also bilipschitz and satisfies $\Lip(f)=\Lip(g)$, $\Lip(f^{-1})=\Lip(g^{-1})$. 
\end{remark}

We define the \emph{isoperimetric profile} of $C$ by
\begin{equation}
\label{eq:profile}
I_C(v)=\inf \Big \{ \pp_C(E) : \Omega \subset C, \vol{E} = v \Big \}.
\end{equation}
We shall say that $E \subset C$ is an \emph{isoperimetric region} if $\pp_C(E)=I_C(\vol{E})$. The \emph{renormalized isoperimetric profile} of $C$ is
\begin{equation}
\label{eq:renprofile}
Y_C=I_C^{(n+1)/n}.
\end{equation}

\begin{lemma}[{\cite[Lemma 5.1]{rv1}}]
\label{lem:link I laC I C}
Let $C$ be a convex body, and $\la> 0$. Then
\begin{equation}
\label{eq:proflac}
I_{\lambda C}(\la^{n+1}v)={\lambda}^nI_{C}(v),
\end{equation}
for all $0< v<\min\{\vol{C},\vol{\la C}\}$.
\end{lemma}

The known results on the regularity of isoperimetric regions are summarized in the following Lemma.

\begin{lemma}[{\cite{MR684753}, \cite{MR862549}, \cite[Thm.~2.1]{MR1674097}}]
\label{lem:n-7}
\mbox{}
Let $C\subset \rr^{n+1}$ a $($possible unbounded$)$  convex body and $E\subset C$ an isoperimetric region.
Then $\ptl E\cap\intt(C) = S_0\cup S$, where  $S_0\cap S=\emptyset$ and
\begin{enum}
\item $S$ is an embedded $C^{\infty}$ hypersurface of constant mean curvature.\item $S_0$ is closed and $H^{s}(S_0)=0$ for any $s>n-7$.
\end{enum}
Moreover, if the boundary of $C$ is of class $C^{2,\alpha}$ then $\cl{\ptl E\cap\intt(C)}=S\cup S_0$, where
\begin{enum}
\item[(iii)] $S$ is an embedded $C^{2,\alpha}$ hypersurface of constant mean curvature
\item[(iv)] $S_0$ is closed and $H^s(S_0)=0$ for any $s>n-7$
\item[(v)] At points of $ S \cap \ptl C$,  $ S$ meets $\ptl C$ orthogonally.
\end{enum}
\end{lemma}


Arguing similarly as in the proof of \cite[Thm.~4.1]{rv1} we obtain
\begin{lemma}
\label{lem:niceapprox}
Let $\{C_i\}_{i\in\nn}$ be a sequence of $($possibly unbounded$)$ convex bodies converging to a convex body $C$ in pointed Hausdorff distance. Let $E\subset C$ a bounded set of finite perimeter and volume $v>0$. If $v_i\to v$. Then there exists a sequence $\{E_i\}_{i\in\nn}$ of bounded sets $E_i\subset C_i$ of finite perimeter in $C_i$ with $\vol{E_i}=v_i$ and $\lim_{i\to \infty} \pp_{C_i}(E_i)=\pp_{C}(E)$. 
\end{lemma}

\begin{proof}
Let $B\subset \rr^{n+1}$ be a closed Euclidean ball containing $E$ in its interior.  By hypothesis, the sequence $\{C_i\cap B\}_{i\in\nn}$ converges in Hausdorff distance to $C\cap B$. As in \cite[Thm.~3.4]{rv1}, we consider a sequence $f_i:C_i\cap B\to C\cap B$ of bilipschitz maps with $\Lip(f_i)$, $\Lip(f_i^{-1})\to 1$. Now we argue as in \cite[Thm.~4.1]{rv1}, defining the sets $E_i\subset C_i$ as the preimages by $f_i$ of smooth perturbations of $E$ supported in the regular part of $\ptl_C E$, and such that $\vol{E_i}=v_i$, and $\lim_{i\to \infty} \pp_{C_i}(E_i)= \pp_{C}(E)$.
\end{proof}

\begin{proposition}[{\cite[Proposition 6.2]{rv1}}]
\label{prp:ICleICmin}
Let $C\subset\rr^{n+1}$ be a convex body $($possibly unboun\-ded$)$, and $p\in \ptl C$. Then every intrinsic ball in $C$ centered at $p$ has no more perimeter than an intrinsic ball of the same volume in $C_p$. Consequently
\begin{equation}
\label{eq:ICleICmin}
I_C(v)\le I_{C_{p}}(v),
\end{equation}
for all $0< v<\vol{C}$.
\end{proposition}

\begin{remark}
\label{rem:half-plane}
A closed half-space $H\subset\rr^{n+1}$ is a convex cone with the largest possible solid angle. Hence, for any convex body $C\subset\rr^{n+1}$, we have
\begin{equation}
\label{eq:half-plane}
I_C(v)\le I_H(v),
\end{equation}
for all $0< v<\vol{C}$.
\end{remark}

\begin{remark}
Proposition~\ref{prp:ICleICmin} implies that $E\cap\ptl C\neq\emptyset$ when $E\subset C$ is isoperimetric. Since  in case $E\cap\ptl C$ is empty, then $E$ is an Euclidean ball. Moreover, as the isoperimetric profile of Euclidean space is strictly larger than that of the half-space, a set whose perimeter is close to the the value of the isoperimetric profile of $C$ must touch the boundary of $C$.
\end{remark}

\begin{proposition}[{\cite[Thm.~2.1]{r-r}}]
\label{prp:spliting}
Let $C$ be an unbounded convex body and $v>0$. Then there exists a finite perimeter set $E\subset C$ $($possibly empty$)$, with $\vol{E}=v_1\le v$, $\pp_C(E)=I_C(v_1)$, and a diverging sequence $\{E_i\}_{i\in\nn}$ of finite perimeter sets such that $\vol{E_i}\to v_2$ and $v_1+v_2=v$. Moreover
\begin{equation}
\label{eq:concavprofconi1}
I_C(v)=\pp_C(E)+\lim_{i\to \infty}\pp_C(E_i)
\end{equation}
\end{proposition}

\begin{lemma}
\label{lem:doubling}
Let $C\subset \rr^{n+1}$ be an unbounded convex body. Then $C$ is a doubling metric space with a constant depending only on $n$.
\end{lemma}

\begin{proof}
Let $x\in C$, $r>0$ and $K$ denote the convex cone with vertex $x$ subtended by $\ptl B_C(x,r)$ then
\begin{equation}
\label{eq:doubling}
\begin{split}
\vol{B_C(x,2r)}& = \vol{B_C(x,2r)\setminus B_C(x,r)}+\vol{B_C(x,r)}
\\
&\le  \vol{B_K(x,2r)\setminus B_K(x,r)}+\vol{B_C(x,r)}
\\
&\le  \vol{B_K(x,2r)}+\vol{B_C(x,r)}
\\
&= 2^{n+1} \vol{B_K(x,r)}+\vol{B_C(x,r)}
\\
&= (2^{n+1}+1) \vol{B_C(x,r)}.
\end{split}
\end{equation}
\end{proof}

The next result follows from \cite[6.1]{gr}.
\begin{proposition}
\label{prp:isopbound}
Let $C\subset \rr^{n+1}$ be an unbounded convex body. Then each isoperimetric region in $C$ is bounded.
\end{proposition}

\begin{proof}
 Using the doubling property, Lemma \ref{lem:doubling}, and \eqref{eq:isnqgdbl1} as in \cite[Lemma 3.10]{gr}, we get  an $c_0>0$ such that
\begin{equation}
\label{eq:exicyl1}
\pp_C(\Om)\ge c_0 \vol{\Om}^{n/(n+1)}
\end{equation}
for any finite perimeter set with $\vol{\Om}\le v_0$.

Let $E\subset C$ be an isoperimetric region so that the regular part of the boundary has constant mean curvature $H$. Consider a point $p$ in the regular part of $\ptl E\cap\intt(C)$, and take a vector field in $\rr^{n+1}$ with compact support in a small neighborhood of $p$ that does not intersect the singular set of $\ptl E$. We choose the vector field so that the deformation $\{E_t\}_{t\in\rr}$ induced by the associated flow strictly increases the volume in the interval $(-\eps,\eps)$, i.e., $t\mapsto\vol{E_t}$ is strictly increasing in $(-\eps,\eps)$. Taking a smaller $\eps$ if necessary, the first variation formulas of volume and perimeter imply 
\begin{equation}
\label{eq:almgren}
\big|\hh^n(\ptl E_t\cap\intt(C))-\hh^n(\ptl E\cap\intt(C)\big|\le (2|H|)\,\big|\vol{E_t}-\vol{E}\big|.
\end{equation}
The last equation plays the role of deformation Lemma in  \cite[Lemma 4.6]{gr},  which combined with \eqref{eq:exicyl1} give us the boundedness of isoperimetric regions.
\end{proof}

We shall say that a cone is \emph{regular} if its boundary is $C^2$ out of the vertices.

\begin{proposition}
\label{prp: asympt ineq cone}
Let $C$ be a regular convex cone  and $\{E_i\}_{i\in\nn}\subset C$ a diverging sequence of finite perimeter sets with $\lim_{i\to\infty} \vol {E_i}=v$. Then $\liminf_{i\to\infty} P_C(E_i)\ge I_{H}(v)$.
\end{proposition}

\begin{proof}
The proof is modeled on \cite[Thm.~3.4]{r-r}, where the sets of the diverging sequence were assumed to have the same volume.
\end{proof}

\section{Unbounded convex bodies with non-degenerate asymptotic cone}
\label{sec:nondegenerate}

The main result in this Section is Theorem~\ref{thm:optimaconi}, where we prove that the isoperimetric profile $I_C$ of an unbounded convex body $C$ with non-degenerate asymptotic cone $C_\infty$ is bounded from below by $I_{C_\infty}$ and that $I_C$ and $I_{C_\infty}$ are asymptotic functions. We also prove the continuity of the isoperimetric profile $I_C$.

Assume now that $C\subset {\rr}^{n+1}$ is an unbounded convex body and $0\in C$. We  denote
\[
C_r =\clb_C(0,r)
\]
 and
\[
 I_{C_r}  (v)=\inf \big\{ P_C(E) : E \subset B_r, \vol{E} = v \big \}.
\]

\begin{lemma}
\label{lem:I_C=inf I_B r}
Let $C$ be an unbounded convex body. Then
\begin{equation}
\label{eq:icinfir}
I_C = \inf_{r>0} I_{C_r}.
\end{equation}
\end{lemma}

\begin{remark}
Lemma~\ref{lem:I_C=inf I_B r} implies that, for every volume, there exists a minimizing sequence consisting of bounded sets.
\end{remark}

\begin{proof}
From the definition of $I_{C_r}$ it follows that, for $0<r<s$, we have $I_{C_s}\ge I_{C_r}\ge I_C$ in the common domain of definition. Hence $I_C \le \inf_{r>0}I_{C_r}$.

In order to prove the opposite inequality we will be follow an argument in \cite{r-r}. Fix $v>0$, and let ${\{ E_i \}}_{i\in \nn}$ be a minimizing sequence for volume $v$. This means $\vol{E_i}=v$ and $\lim_{i\to \infty} P_C(E_i)=I_C(v)$.

For every $i\in \nn$ we have $\lim_{r\to \infty}\vol{E_i\setminus B_r}=0$. Thus for every $i\in \nn$ there exists $R_i>0$ such that
\[
\vol{E_i\setminus B_{R_i}}<\frac{1}{i}.
\]
We now define a sequence of real numbers ${\{ r_i \}}_{i\in \nn}$ by induction taking $r_1=R_1$ and $r_{i+1}=\max\{r_i,R_{i+1}+1\}+i$. Then $\{r_i\}_{i\in\nn}$ satisfies
\[
r_{i+1}-r_i \ge i\qquad \mbox{and}\quad \vol{E_i\setminus B_{r_i}}<\frac{1}{i}.
\]
By the coarea formula
\[
\int_{r_i}^{r_{i+1}}{\h}^{n}(E_i \cap \ptl B_t)\, dt\le \int_{\rr}{\h}^{n}(E_i \cap \ptl B_t)\,dt= \vol{E_i}=v.
\]
Thus there exists $ \rho(i)\in [r_i,r_{i+1}]$ so that $(r_{i+1}-r_i)\,{\h}^{n}(E_i \cap \ptl B_{\rho(i)})\le v$, and so
\[
{\h}^{n}(E_i \cap \ptl B_{\rho(i)})\le\frac{v}{i}.
\]
Now by Corollary 5.5.3 in \cite{Ziemer} we have
\[
P_C(E_i \cap B_{\rho(i)})\le P(E_i,B_{\rho(i)}) + {\h}^{n}(E_i \cap
\ptl B_{\rho(i)}).
\]
Let $B_i^*$ be a sequence of Euclidean balls of volume $\vol{B_i^*}=\vol{E_i\setminus B_{\rho(i)}}$. Since $\vol{B_i^*}\to 0$ when $i\to\infty$, the balls can be taken at positive distance of $E_i\cap B_{\rho(i)}$, but inside $B_{2r_i}$ for $i$ large enough. Hence
\begin{align*}
I_{C_{2r_i}}(v)&\le P_C(E_i\cap B_{\rho(i)})+P(B_i^*)
\\
&\le  P_C(E_i,B_{\rho(i)}) + {\h}^{n}(E_i \cap \ptl B_{\rho(i)})+P(B_i^*)
\\
&\le P_C(E_i)+\frac{v}{i}+P(B_i^*).
\end{align*}
Taking limits when $i\to\infty$ we obtain $\inf_{r>0} I_{C_r}(v)\le I_C(v)$.
\end{proof}

The following is inspired by \cite[Thm.~4.12]{rv1}

\begin{lemma}
\label{lemma:isnqgdblconi}
Let $C\subset \rr^{n+1}$ a convex body with non-degenerate asymptotic cone $C_\infty$. Given $r_0>0$, there exist positive constants $M$, $\ell_1$, only depending on $r_0$ and $C_{\infty}$, and a universal positive constant $\ell_2$ so that
\begin{equation}
\label{eq:isnqgdbl1}
I_{\clb_C(x,r)}(v)\ge M\, {\min \{v,\vol{\clb_C(x,r)}-v\}}^{n/(n+1)},\end{equation}
for all $x\in C$, $0<r\le r_0$, and $0<v<\vol{\clb(x,r)}$. Moreover
\begin{equation}
\label{eq:isnqgdbl1a}
\ell_1 r^{n+1} \le \vol{\clb_C(x,r)} \le \ell_2 r^{n+1},
\end{equation}
for any $x\in C$, $0<r\le r_0$.
\end{lemma}

\begin{proof}
Fix $r_0>0$. Following \cite[Thm.~4.11]{rv1}, to show the validity of \eqref{eq:isnqgdbl1}, we only need to obtain a lower bound $\delta$ for the inradius of $\clb_C(x,r_0)$ independent of $x\in C$. Then a relative isoperimetric inequality is satisfied in $\clb_C(x,r)$, for $0<r<r_0$, with a constant $M$ that only depends on $r_0/\de$.

Let $C_\infty$ be the asymptotic cone of $C$ with vertex at the origin. For every $x\in C$, we have $x+C_\infty=\bigcap_{\la>0} h_{x,\la}(C)=\bigcap_{1\ge\la>0} h_{x,\la}(C)\subset C$.  Fix $r_0>0$ and $x\in C$. As $x+C_\infty\subset C$, we get $\clb_{x+C_\infty}(x,r) \subset \clb_C(x,r)$. Since $C_{\infty}$ is non-degenerate, then we can  pick $\de>0$ and $y\in C_\infty$ so that $B(y,\delta)\subset \clb_{C_\infty}(0,r_0)$. Hence $B(x+y,\delta)\subset \clb_{x+C_\infty}(x,r_0)$. This provides the desired uniform lower bound for the inradius of $\clb(x,r_0)$.

%
%

We now prove \eqref{eq:isnqgdbl1a}. Since $\vol{\clb_C(x,r)}\le \vol {\clb(x,r)}$, it is enough to take $\ell_2=\omega_{n+1}=\vol{\clb(0,1)}$. For the remaining inequality, using the same notation as above, we have
\begin{align*}
\vol{\clb(x,r)\cap C}&= \vol{\clb(x,\la r_0)\cap C}\ge \vol{h_{x,\la}(\clb(x,r_0) \cap C)}
\\
&=\la^{n+1} \vol{\clb(x,r_0)\cap C}\ge \la^{n+1} \vol{\clb(y(x),\de)}
\\
&=\omega_{n+1}(\de/r_0)^{n+1}\,r^{n+1},
\end{align*}
and we take $\ell_1=\omega_{n+1}(\de/r_0)^{n+1}$.
\end{proof}

\begin{theorem}
\label{thm:optimaconi}
Let $C$ be a convex body with non-degenerate asymptotic cone $C_\infty$. Then
\begin{equation}
\label{eq:icgeicinfty}
\frac{I_C}{I_{C_{\infty}}}\ge 1.
\end{equation}
Moreover
\begin{equation}
\label{eq:optimanondegen}
\lim_{v\to\infty}\frac{I_C(v)}{I_{C_{\infty}}(v)}=1.
\end{equation}
\end{theorem}

\begin{proof}
Fix $v>0$ and let $E\subset C$ be any bounded set of finite perimeter and volume $v$.
Let $q\in \intt(C_{\infty}\cap\clb(0,1))$ and  $B_q\subset \intt(C_{\infty}\cap\clb(0,1))$ be a  Euclidean geodesic ball. Now  consider a solid cone $K_q$ with vertex $q$ such that $0\in \intt(K_q)$ and $K_q\cap C\cap \ptl B(0,1)=\emptyset$.
Let $r_i \uparrow\infty$. By definition of the asymptotic cone,
$r_i^{-1} C\cap \clb(0,1)$ converges to $ C_{\infty} \cap \clb(0,1)$ in Hausdorff distance. Thus we may construct, as in \cite[Thm.~3.4] {rv1}, a family of bilipschitz maps $f_i:r_i^{-1}C\cap \clb(0,1)\to C_{\infty}\cap\clb(0,1)$ which fix the points in the ball $B_q$, and such that 
\begin{equation}
\label{eq:optimanondegen1}
\Lip(f_i), \Lip(f^{-1}_i)\to 1.
\end{equation}
So $f_i$ is the identity in $B_q$ and it is extended linearly along the segments leaving from $q$. For large enough $i\in \nn$ we have, $E\subset C\cap B(0,r_i)$ and $r_i^{-1}E\subset K_q$, since $\diam(E)<\infty$. For this large $i$, by construction, the maps $f_i$ have the additional property
\begin{equation}
\label{eq:optimanondegen2}
\pp_{C_{\infty}}(f_i(r_i^{-1}E))=\pp_{C_{\infty}\cap\clb(0,1)}(f_i(r_i^{-1}E)).
\end{equation}
For $i$  large enough, $\pp_C(E)=\pp_C(E\cap B(0,r))$. Thus by Lemma \ref{lem:bilip}, \eqref{eq:isopkon}  and the above, we get
\begin{equation}
\label{eq:optimanondegen4}
\begin{split}
\frac{\pp_{C}(E)}{\vol{E}^{n/(n+1)}}=&\frac{\pp_{r_i^{-1}C}(r_i^{-1}E)}{\vol{r_i^{-1}E}^{n/(n+1)}}\ge \frac{\pp_{C_\infty}(f_i(r_i^{-1}E))}{\vol{f_i(r_i^{-1}E)}^{n/(n+1)}}\,(\Lip(f_i)\Lip(f_i^{-1}))^{-n}
\\
\ge &\,I_{C_{\infty}}(1)\,(\Lip(f_i)\Lip(f_i^{-1}))^{-n}.
\end{split}
\end{equation}
Passing to the limit we get,
\begin{equation}
\label{eq:optimanondegen5}
\frac{\pp_C(E)}{\vol{E}^{n/(n+1)}}\ge I_{C_{\infty}}(1).
\end{equation}
Thus, by \eqref{eq:isopkon}, for every $v\ge 0$, we obtain,
\begin{equation}
\label{eq:optimanondegen6}
I_C(v)\ge I_{C_{\infty}}(v),
\end{equation}
which implies \eqref{eq:icgeicinfty}.

Let us prove now \eqref{eq:optimanondegen}. Let $\la_i\downarrow 0$, $i\in\nn$. Since $C_{\infty}$ is the asymptotic cone of each $\la_i C$ then the last inequality holds for every  $\la_i C, i\in\nn$. Passing to the limit we conclude
\[
I_{C_{\infty}}(1)\le\liminf_{i\to\infty}I_{\la_i C}(1).
\]
Now  consider a ball  $B_{C_{\infty}}$ centered at a vertex of $C_{\infty}$ of volume $1$,  which is an isoperimetric region by \cite{lions-pacella}. By Lemma \ref{lem:niceapprox}, there exist a sequence $E_i\subset\la_i C$ of finite perimeter sets with  $\vol{E_i}=1$ and such that $\lim_{i\to \infty} \pp_{\la_iC}(E_i)=\pp_{C}(B)$. So we get
\[
I_{C_{\infty}}(1)\ge\limsup_{i\to\infty}I_{\la_i C}(1),
\]
and we conclude
\begin{equation}
\label{eq:optimanondegen7}
I_{C_{\infty}}(1)=\lim_{i\to\infty}I_{\la_i C}(1).
\end{equation}
From \eqref{eq:optimanondegen7}, Lemma \ref{lem:link I laC I C} and the fact that $C_{\infty}$ is a cone we deduce
\[
1=\lim_{\la\to 0}\frac{I_{\la C}(1)}{I_{C_{\infty}}(1)}=
\lim_{{\la}\to
0}\frac{{\la}^nI_C(1/{\la}^{n+1})}{{\la}^n I_{C_{\infty}}(1/{\la}^{n+1})}=\lim_{v\to\infty}\frac{I_C(v)}{I_{C_{\infty}}(v)},
\]
as desired.
\end{proof}

We now prove the continuity of the isoperimetric profile of $C$. The proof of the following is adapted from \cite[Lemma 6.2]{gallot}

\begin{lemma}
\label{lem:contprofconi}
Let $C$ be a convex body with non-degenerate asymptotic cone. Then $I_C$ is continuous.
\end{lemma}

\begin{proof}
Given $r>0$ and $x\in C$, we get $B(x,r)\cap (x+C_{\infty})\subset B(x,r)\cap C$. Thus
\[
\vol{B_C (x,r)}\ge \vol{B_{x+C_{\infty}} (x,r)}= \vol{B_{x+C_{\infty}} (x,1)}\,r^{n+1}=\ell_1r^{n+1},
\]
for all $x\in C$ and $r>0$, where $\ell_1=|B_{C_\infty}(0,1)|$.

Let $E \subset C $ a finite perimeter set and $r>0$.  We apply Fubini’s Theorem to the function $C \times E \to \rr$ defined by
\[
(x, y) \mapsto \chi_{B_C (x,r)} ( y)
\]
to obtain
\[
\int_C\vol{B_C (x,r)\cap E}\,dx=\int_E\vol{B_C (y,r)}\,dy \ge \ell_1r^{n+1}\vol{E}.
\]
This implies the existence of some $x\in C$ (depending on $E$ and $r>0$) such that
\begin{equation}
\label{eq:contprofconi1}
\vol{B_C (x,r)\cap E}\ge \ell_1 r^{n+1}\frac{\vol{E}}{\vol{C}}.
\end{equation}

Fix now two volumes $0<v_1<v_2$. Define $r > 0$ by
\[
\ell_1 r^{n+1}\frac{v_2}{\vol{C}}=v_2-v_1.
\]
Fix $\eps>0$. From the definition of the isoperimetric profile, there exists a finite perimeter set $E\subset C$ of volume $v_2$ such that $\pp_C(E)\le I_C(v_2)+\eps$ . From the above discussion, there exists $x\in C$ so that \eqref{eq:contprofconi1} holds. This implies
\[
\vol{E \setminus B_C (x,r)}\le \vol{E}-\vol{B_C (x,r)\cap E}\le v_2 -\ell_1 r^{n+1}\frac{v_0}{\vol{C}}=v_1.
\]
As the function $t\mapsto \vol{E \setminus B_C (x,t)}$ is continuous and monotone, there exists $0 < s \le  r $ so that $ \vol{E \setminus B_C (x,s)} =v_1$ . Hence we get
\begin{align*}
I_C (v_1)&\le \pp_C(E \setminus B_C (x,s))\le  \pp_C(E)+ \pp_C( B_C (x,s))
\\
&\le I_C (v_2)+\eps+m s^n\le I_C (v_2)+\eps+m r^n 
\\
&\le I_C (v_2)+\eps+c\,v_1^{-n/(n+1)}(v_2-v_1)^{n/(n+1)},
\end{align*}
where $m>0$ is the perimeter of a Euclidean geodesic sphere of radius 1 and $C>0$ is explicitly computed from the definition of $r$. As $\eps$ was arbitrary, we get
\begin{equation}
\label{eq:left}
I_C(v_1)\le I_C(v_2)+c\,v_1^{-n/(n+1)}(v_2-v_1)^{n/(n+1)}.
\end{equation}

We now prove a second inequality. By Lemma \ref{lem:I_C=inf I_B r}, given $\eps>0$,  there exists $R>0$ and a finite perimeter set $E\subset \clb_C(0,R)$ of volume $v_0$ such that $\pp_C(E)\le I_C(v_1)+\eps$ . Now consider a Euclidean geodesic ball $B$ of volume $v_2-v_1$ in $\intt(C)\setminus \clb(0,R))$. We have
\[
I_C(v_2)\le P_C(E\cup B)=P_C(E)+P_C(B)\le I_C(v_1)+\eps+c\,(v_2-v_1)^{n/(n+1)},
\]
where $c'>0$ is the Euclidean isoperimetric constant. Since $\eps>0$ is arbitrary, we get
\begin{equation}
\label{eq:right}
I_C(v_2)\le I_C(v_1)+c'\,(v_2-v_1)^{n/(n+1)}.
\end{equation}

Now the continuity of $I_C$ follows from \eqref{eq:left} and \eqref{eq:right}.
\end{proof}

\section{Conically bounded convex bodies}
\label{sec:conically}

In this Section we shall obtain a number of results for conically bounded convex bodies with smooth boundary. Observe that this assumption does not guarantee that the asymptotic cone has smooth boundary out of the vertexes: simply consider the function in $\rr^2$ defined by $f(x,y)=(1+x^2)^{1/2}+(1+y^2)^{1/2}$. The asymptotic cone of its epigraph can be computed as in the proof of Lemma~\ref{lem:C_infty=C^infty} as $\{(x,y,z)\in\rr^3:z\ge f_\infty(x,y)\}$, where $f_\infty$ is the limit, when $\mu\to\infty$, of the functions $f_\mu(p)=\mu^{-1}f(\mu p)$. In our case, $f_\infty(x,y)=|x|+|y|$.

We shall say that a conically bounded convex body is \emph{regular} if it has smooth boundary and its asymptotic cone has smooth boundary out of the vertexes.

The following elementary result on convex functions will be needed

\begin{lemma}
\label{lem:convexasymp}
Let $a>0$, and $f:[0,+\infty)\to [0,+\infty)$ a convex function~satisfying
\[
\lim_{x\to\infty} f(x)-(ax+b)=0.
\]
Then, for every $x_0\ge 0$ and any $u_0\ge f(x_0)$, the halfline $\{(x,u_0+a\,(x-x_0)):x\ge x_0\}$ is contained in the epigraph of $f$.
\end{lemma}

\begin{proof}
Let us prove first that the function $x\mapsto (x-x_0)^{-1}(f(x)-u_0)$ is non-decreasing. Let $x_0<x<z$ so that $x=x_0+\la\,(z-x_0)$, with $\la=(x-x_0)/(z-x_0)$. By the concavity of $f$ we get $f(x)=f(\la\,z+(1-\la)\,x_0)\le \la\,f(z)+(1-\la)\,f(x_0)\le \la\,f(z)+(1-\la)\,u_0$. Hence $f(x)-u_0\le\la\,(f(z)-u_0)$, what implies
\[
\frac{f(x)-u_0}{x-x_0}\le\frac{f(z)-u_0}{x-x_0},
\]
as we claimed.

For any $x>x_0$, the segment joining the points $(x_0,u_0)$ and $(x,f(x))$ is contained in the epigraph of $f$ by the concavity of $f$. Moreover, we have
\[
\frac{f(x)-u_0}{x-x_0}\le\frac{f(x)-f(x_0)}{x-x_0}=
\frac{f(x)-ax-b}{x-x_0}-\frac{f(x_0)-ax-b}{x-x_0},
\]
and taking limits we get
\[
\lim_{x\to\infty}\frac{f(x)-u_0}{x-x_0}\le a,
\]
by the monotonicity of $x\mapsto (x-x_0)^{-1}(f(x)-u_0)$ and the asymptotic property of the line $ax+b$. So we conclude $f(x)-u_0\le a\,(x-x_0)$ for all $x>x_0$, as claimed.
\end{proof}

\begin{proposition}
\label{prp:asimpineqconi}
Let $C$ be a regular conically bounded convex body, and $\{E_i\}_{i\in\nn}$ a diverging sequence of finite perimeter sets with $\lim_{i\to\infty}\vol{E_i}=v$. Then,
\[
\liminf_{i\to\infty} P_C(E_i)\ge I_H(v),
\]
where $I_H$ is the isoperimetric profile of a closed half-space $H$.
\end{proposition}

\begin{proof}
Assume that $0$ is the vertex of $C_\infty=C^\infty$. As usual, let $C_{s}=C\cap\{x_{n+1}=s\}$.  The orthogonal projection of $\rr^{n+1}$ over $\{x_{n+1}=0\}$ will be denoted by $\pi$. The balls considered in what follows will be $n$-dimensional.



For $t_0>0$ take a positive radius $r_0>0$ so that $B(0,r_0)\times\{t_0\}\subset\intt C_{t_0}$. Is is an easy consequence of Lemma~\ref{lem:convexasymp} that the cone of base $B(0,r_0)\times\{t_0\}$ with vertex $0$, intersected with $t\ge t_0$, is contained in the interior of $C$. The section of this cone at height $t$ is $B(0,tr_0/t_0)\times\{t\}$, and so $B(0,tr_0/t_0)\subset\intt \pi(C_t)$.

We define $F:C\cap\{t\ge t_0\}\to C_\infty\cap\{t\ge t_0\}$ by
\[
F(x,t)=(\tilde{f}_t(x),t),
\]
where $\tilde{f}_t:\pi(C_t)\to\pi((C_\infty)_t)$ is the map defined by equation (3.6) in \cite{rv1} which leaves fixed the~points in the inner ball $B(0,tr_0/t_0)\subset\intt\pi(C_t)$. For $i\ge t_0$, let $F_i=F|_{C\cap\{x_{n+1}\ge i\}}$.

Let us denote by $h_\la$ the dilation in $\rr^n$ of ratio $\la>0$. Taking $\la=t_0/t$ we have
\[
B(0,r_0)=h_\la(B(0,\tfrac{t}{t_0}r_0))\subset\intt h_\la(\pi(C_t))\subset\intt h_\la(\pi((C_\infty)_t))=\intt \pi((C_\infty)_{t_0}).
\]
When $t\to\infty$, $h_\la(\pi(C_t))\to\pi((C_\infty)_{t_0})$ in Hausdorff distance since $C_\infty$ is the asymptotic cone of $C$. Let $f_t:h_\la(\pi(C_t))\to\pi((C_\infty)_{t_0})$ be the family of maps given by (3.6) in \cite{rv1} leaving fixed the ball $B(0,r_0)$ so that $\Lip(f_t)$, $\Lip(f_t^{-1})\to 1$. It is immediate to show that $\tilde{f}_t=h_{\la^{-1}}\circ f_t\circ h_\la$ and that $\Lip(\tilde{f}_t)=\Lip(f_t)$, $\Lip(\tilde{f}_t^{-1})=\Lip(f_t^{-1})$. We conclude that $\Lip(\tilde{f}_t)$, $\Lip(\tilde{f}_t^{-1})\to 1$.

Let $t\ge s\ge i\ge t_0$. We estimate
\begin{equation}
\label{eq:fxt2}
\begin{split}
|F(x,t)-F(y,s)|&=\big(|\tilde{f}_t(x)-\tilde{f}_s(y)|^2+|t-s|^2\big)^{1/2}
\\
&=\big(|\tilde{f}_t(x)-\tilde{f}_t(y)+\tilde{f}_t(y)-\tilde{f}_s(y)|^2+|t-s|^2\big)^{1/2}
\\
&=\big(|\tilde{f}_t(x)-\tilde{f}_t(y)|^2+|\tilde{f}_t(y)-\tilde{f}_s(y)|^2
\\
&\qquad+2\,|\tilde{f}_t(x)-\tilde{f}_t(y)||\tilde{f}_t(y)-\tilde{f}_s(y)|+|t-s|^2\big)^{1/2}.
\end{split}
\end{equation}

We have $|(\tilde{f}_t(x)-\tilde{f}_t(y))|\le\Lip(\tilde{f}_t)|x-y|$. By  \cite[Thm.~3.4]{rv1}, we can write $\Lip(\tilde{f}_t)<(1+\eps_i)$ for $t\ge i$, where $\eps_i\to 0$ when $i\to\infty$. Hence
\begin{equation}
\label{eq:eksis f t-f s < Lip cyl}
|\tilde{f}_t(x)-\tilde{f}_t(y)|\le(1+\eps_i)\,|x-y|, \qquad\text{for}\  t\ge i.
\end{equation}

We estimate now $|\tilde{f}_t(y)-\tilde{f}_s(y)|$.

In case $|y|\le sr_0/t_0\le tr_0/t_0$, we trivially have $|\tilde{f}_t(y)-\tilde{f}_s(y)|=0$. Let us consider the case $|y|\ge tr_0/t_0\ge sr_0/t_0$. Set $u=y/|y|$ and for every $t>0$ denote $\rho_t(u)=\rho(C_t,u)$, $\tilde{\rho_t}(u)=\rho ((C_{\infty})_t,u)$
hence by (3.7) in \cite[Thm.~3.4]{rv1} we have
\begin{equation}
\label{eq:asumpineqconi4}
\begin{split}
|\tilde{f}_t(y)-\tilde{f}_s(y)|&=\Big|\frac{(tr_0/t_0-|y|)}{\tilde{\rho}_t(u)-tr_0/t_0}\big(\tilde{\rho}_t(u)-\rho_t(u)\big)
-\frac{(sr_0/t_0-|y|)}{\tilde{\rho}_s(u)-sr_0/t_0}\,\big(\tilde{\rho}_s(u)-\rho_s(u)\big)\Big|
\\
&\le\frac{|sr_0/t_0-|y||}{|\tilde{\rho}_s(u)-sr_0/t_0|}\big|(\tilde{\rho}_t(u)-\rho_t(u))
-(\tilde{\rho}_s(u)-\rho_s(u)\big)\big|
\\
&+\big|(\tilde{\rho}_t(u)-\rho_t(u))\big| \Big| \frac{tr_0/t_0-|y|}{\tilde{\rho}_t(u)-tr_0/t_0}-\frac{s r_0/t_0-|y|}{\tilde{\rho}_s(u)-sr_0/t_0}\Big|
\\
&\le\big|(\tilde{\rho}_t(u)-\rho_t(u))-(\tilde{\rho}_s(u)-\rho_s(u)\big)\big|
\\
&\qquad\qquad\qquad+M\Big| \frac{tr_0/t_0-|y|}{\tilde{\rho}_t(u)-tr_0/t_0}-\frac{sr_0/t_0-|y|}{\tilde{\rho}_s(u)-sr_0/t_0}\Big|,
\end{split}
\end{equation}
where we have used
\[
\frac{|sr_0/t_0-|y||}{|\tilde{\rho}_s(u)-sr_0/t_0|}\big|\le 1,
\]
since $|y|\le \tilde{\rho}_s(u)$ (because $y\in\pi(C_s)\subset\pi((C_{\infty})_s) $), and $\big|(\tilde{\rho}_t(u)-\rho_t(u))\big|\le M$ for $t>1$, since $\sup_{u\in\esf^{n-1}}|\tilde{\rho}_t(u)-\rho_t(u)|\to 0$ and so that $M$ does not depend on $i, u$.
For $u\in\esf^{n-1}$, consider the functions $\rho_t(u)=\rho(C_t,u)$, $\tilde{\rho}_t(u)=\rho ((C_{\infty})_t,u)$.  Observe that, for every $u\in \esf^{n}$ orthogonal to $\ptl/\ptl x_{n+1}$, the 2-dimensional half-plane defined by $u$ and $\ptl/\ptl x_{n+1}$ intersected with $C$ is a 2-dimensional convex set, and the function $t\mapsto \rho_t(u) $ is concave with asymptotic line the function $t\mapsto \tilde{\rho}_t(u)$. Thus the function $t\mapsto\rho_t(u)-\tilde{\rho}_t(u)$ is concave, because $t\mapsto {\rho}_t(u)$ is concave and $t\mapsto \tilde{\rho}_t(u)$ is affine, and so
\begin{equation}
\label{eq:asumpineqconi5}
\begin{split}
\frac{\big|(\tilde{\rho}_t(u)-\rho_t(u))-(\tilde{\rho}_s(u)-\rho_s(u)\big)\big|}{|t-s|}\le
\big|(\tilde{\rho}_i(u)-\rho_i(u))-(\tilde{\rho}_{i-1}(u)-\rho_{i-1}(u)\big)\big|.
\end{split}
\end{equation}
Thus by \eqref{eq:def ex asymp con}, the lipschitz constant of $t\mapsto (\tilde{\rho}_t(u)-\rho_t(u))|_{\{t\ge i\}}$ is independent of $u$ and tends to $0$ as $i\to+\infty$.
So, only remains to estimate the second term in the right part of \eqref{eq:asumpineqconi4}.
To accomplish that, set
\[
\rho(u)=\rho((C_\infty)_{t_0},u)=\rho(h_{t_0/{t}}(\pi((C_\infty)_{t}),u)\quad\text{for every}\ u\in\esf^{n-1}.
\]
By the homogeneity of the radial function we get
\[
\rho(u)=\frac{t_0}{t}\rho(\pi((C_\infty)_{t}),u)=\frac{t_0}{t}\tilde{\rho}_{t}(u) \quad\text{for every}\ t\ge t_0.
\]
Consequently if $R$ is the inradius of
$(C_{\infty})_{t_0}$, and $u_0$ such that $\rho(u_0)=\min_{u\in\esf^{n-1}}\rho(u)$, then
\begin{equation}
\label{eq:asumpineqconi5a}
\begin{split}
\Big| \frac{tr_0/t_0-|y|}{\tilde{\rho}_t(u)-tr_0/t_0}-\frac{sr_0/t_0-|y|}{\tilde{\rho}_s(u)-sr_0/t_0}\Big|
&\le\Big| \frac{tr_0/t_0-|y|}{t/t_0\,\tilde{\rho}(u)-tr_0/t_0}-\frac{sr_0/t_0-|y|}{s/t_0\,\tilde{\rho}(u)-sr_0/t_0}\Big|
\\
&\le\frac{|y|t_0}{{\rho}(u)-r_0}\big|\frac{1}{t}-\frac{1}{s}\big|
\\ &\le\frac{Rt_0}{{\rho}(u_0)-r_0}\big|\frac{1}{t}-\frac{1}{s}\big|
\\ &\le\frac{Rt_0}{{\rho}(u_0)-r_0}\,\frac{1}{i^2}|t-s|
\end{split}
\end{equation}
Thus, the lipschitz constant of
\[
t\mapsto \frac{tr_0/t_0-|y|}{\tilde{\rho}_t(u)-tr_0/t_0}\Big|_{\{t\ge i\}}
\]
is independent of $u$ and tends to $0$ as $i\to+\infty$.

By the above discussion and \eqref{eq:asumpineqconi4}, there exists $\ell_i$ for every $i\in\nn$ such that $\ell_i\to 0$, and
\begin{equation}
\label{eq:estimation f_t(y)-f_s(y) final2}
|f_t(y)-f_s(y)|\le\ell_i\,|t-s|.
\end{equation}
From \eqref{eq:fxt2}, \eqref{eq:eksis f t-f s < Lip cyl}, \eqref{eq:estimation f_t(y)-f_s(y) final2}, and trivial estimates, we obtain
\begin{equation}
\label{eq:eksis F(x,t)-F(y,s) telik}
|F_i(x,t)-F_i(y,s)|\le \big((1+\eps_i)^2+\ell_i^2+(1+\eps_i)\,\ell_i\big)^{1/2}\,|x-y|
\end{equation}

Now $\eps_i\to 0$ and $\ell_i\to 0$ as $i\to\infty$. Thus inequality \eqref{eq:eksis F(x,t)-F(y,s) telik} finally give us
\[
\limsup_{i\to\infty}\Lip(F_i)\le 1.
\]
Similarly we find $\limsup_{i\to\infty}\Lip(F_i^{-1})\le 1$. From the general inequality $\Lip(F_i^{-1})\Lip(F_i)\ge 1$ we finally get that $\max\{\Lip(F_i),\Lip(F_i^{-1})\}\to 1$ when $i\to\infty$ (indeed we have just proved that $d_L(C\cap \{ x_{n+1}\ge i \},C^{\infty}\cap \{ x_{n+1}\ge i \})\to 0$).

Now in case that $|y|\ge tr_0/t_0$ but $|y|\le sr_0/t_0$, we can find $t^*>0$ such that $|y|= t^*r_0/t_0$, then as $\tilde{f}_t(y)=\tilde{f}_{t^*}(y)=y$, but in the same time $\tilde{f}_{t^*}(y)$ can have the expression of (3.7) in \cite[Thm.~3.4]{rv1} then after a triangle inequality argument this case is reduced to the previous one.
%
\end{proof}

\begin{proposition}
\label{prp: exist isop reg conical}
Let $C$ be a regular conically bounded convex body. Then isoperimetric regions exist in $C$ for all volumes.
\end{proposition}

\begin{proof}
Fix $v>0$. By Proposition~\ref{prp:spliting}, there exists $E\subset C$ (possibly empty) such that $\vol{E}=v_1$, $\pp_C(E)=I_C(v_1)$, and a diverging sequence $\{E_i\}_{i\in\nn}$ of finite perimeter sets such that $\vol{E_i}\to v_2=v-v_1$; moreover
\begin{equation}
\label{eq:pp_C(E)+lim pp_C(E_i)_}
I_C(v)=\pp_C(E)+\lim_{i\to \infty}\pp_C(E_i)
\end{equation}

Assume now that $v_2>0$. From Proposition~\ref{prp:asimpineqconi} we get $\lim \pp_C(E_i)\ge I_H(v_2)$. Now by Proposition~\ref{prp:isopbound}, the set $E$ is bounded and by Proposition~\ref{prp:ICleICmin} we can find an intrinsic ball $B\subset C$ with volume $v_2$ such that $E\cap B=\emptyset$ and $\pp_C(B)\le I_H(v_2)$. Then \eqref{eq:pp_C(E)+lim pp_C(E_i)_} gives
\begin{equation}
I_C(v)=\pp_C(E)+\lim_{i\to \infty}\pp_C(E_i)\ge \pp_C(E)+I_H(v_2)\ge \pp_C(E)+\pp_C(B).
\end{equation}
Thus $E\cup B$ is an isoperimetric region with volume $v$. 
\end{proof}

\begin{proposition}
\label{prp: I C is concave conical}
Let $C\subset\rr^{n+1}$ be a conically bounded convex set. Then $I_C,Y_C$ are positive concave functions, and so they are non-decreasing. Consequently, every isoperimetric region in $C$ is connected.
\end{proposition}

\begin{proof}
By \ref{prp: exist isop reg conical} isoperimetric regions exist for all volumes thus we can argue as in \cite[Thm.~3.2]{bay-rosal} to conclude that the upper second derivative of $Y_C$ is non-positive, where combining with the fact that $Y_C$ is continuous \ref{lem:contprofconi}, we deduce that $Y_C$ is concave. And so is $I_C$ as a composition of non-negative concave functions.

The connectedness of the isoperimetric regions is an implication of the concavity of $Y_C$,  \cite[Thm.~4.6]{rv1}.
\end{proof}

\begin{corollary}
\label{cor:minimizing sequences}
Let $C\subset\rr^{n+1}$ be a regular conically bounded convex body. Given any $v>0$, any minimizing sequence for volume $v$ converges to an isoperimetric region.
\end{corollary}

\begin{proof}
We reason by contradiction as in the proof of Proposition~\ref{prp: exist isop reg conical}. Then we find an isoperimetric region in $C$ consisting of two components $E$ and $B$, a contradiction to Proposition~\ref{prp: I C is concave conical}.
\end{proof}

As a consequence we have the two following lemmata, \cite{rv1}
\begin{lemma}
\label{lem:I_C(v)> cv}
Let $C\subset \rr^{n+1}$ be a regular conically bounded convex body, and $0<v_0<\vol{C}$. Then
\begin{equation}
\label{eq:I_C(v)> cv.}
I_C(v)\ge \frac{I_C(v_0)}{v_0^{n/(n+1)}}\,v^{n/(n+1)},
\end{equation}
for all $0< v\le v_0$. 
\end{lemma}

\begin{lemma}
\label{lem:_IC<la C}
Let $C$ be a regular conically bounded convex body, $\la\ge 1$. Then
\begin{equation}
\label{eq:ilacic}
I_{\la C}(v)\ge I_C(v),
\end{equation}
for all $0< v<\vol{C}$.
\end{lemma}



In a similar way to \cite[p.~18]{le-ri}, given a convex body (possibly unbounded) $C$ and $E\subset C$, we define a function $h:C\times(0,+\infty)\to (0,\tfrac{1}{2})$ by
\begin{equation}
\label{eq:defh}
h(E,C,x,R)=\frac{\min\big\{ \vol{E\cap B_C(x,R)},\vol{B_C(x,R)\setminus E} \big\}}{\vol{B_C(x,R)}},
\end{equation}
for $x\in C$ and $R>0$. When $E$ and $C$ are fixed, we shall simply denote
\begin{equation}
h(x,R)=h(E,C,x,R).
\end{equation}

For future reference, we state the following result

\begin{lemma}[{\cite[Lemma~5.4]{rv1}}]
\label{lem:fv}
For any $v>0$, consider the function $f_v:[0,v]\to\rr$ defined by
\begin{equation*}
f_v(s)=s^{-n/(n+1)}\,\bigg(\bigg(\frac{v-s}{v}\bigg)^{n/(n+1)}-1\bigg).
\end{equation*}
Then there is a constant $0<c_2<1$ that does not depends on $v$ so that $f_v(s)\ge -(1/2)\,v^{-n/(n+1)}$ for all $0\le s\le c_2\,v$.
\end{lemma}

Now we can prove the following density estimate. A similar estimate for convex bodies was obtained in \cite[Prop.~4.9]{rv1}.

\begin{proposition}
\label{prp:leon rigot lem 42 coni}
\mbox{}
Let $C\subset \rr^{n+1}$ be a regular conically bounded convex body, and $E\subset C$ an isoperimetric region of volume $0<v<|C|$.  Choose $\eps$ so that
\begin{equation}
\label{eq:epsfine}
0<\eps<\min\bigg\{{\ell_2}^{-1}v,c_2v,\frac{\ell_2^{n}}{8^{n+1}},{\ell_2}^{-1}  \bigg(\frac{c_1}{4}\bigg)^{n+1}\bigg\}  \bigg\},
\end{equation}
where $c_1=v^{-n/(n+1)}I_C(v)$ and $c_2$ is the constant in Lemma~\ref{lem:fv}.

Then, for any $x\in C$ and $R\le 1$ so that $h(x,R)\le\eps$, we get
\begin{equation}
\label{eq:hr/2=0}
h(x,R/2)=0.
\end{equation}
Moreover, in case $h(x,R)=\vol{E\cap B_C(x,R)}\vol{B_C(x,R)}^{-1}$, we get $|E\cap B_C(x,R/2)|=0$ and, in case $h(x,R)=\vol{B_C(x,R)\setminus E}\vol{B_C(x,R)}^{-1}$, we have $|B_C(x,R/2)\setminus E|=0$.
\end{proposition}

\begin{proof}
In case $h(x,R)=\vol{E\cap B_C(x,R)}\vol{B_C(x,R)}^{-1}$ we argue as in \cite[Prop.~4.9]{rv1} to get 
\[
bR/4a\le(m(R)^{1/{(n+1)}}-m(R/2)^{1/{(n+1)}})\le m(R)^{1/{(n+1)}}\le (\eps\ell_2)^{1/(n+1)}R.
\]
This is a contradiction, since $\eps\ell_2<(b/4a)^{n+1}=I_C(v)^{n+1}/(8^{n+1} v^n)\le \ell_2^{n+1}/8^{n+1}$ by \eqref{eq:epsfine} and Proposition~\ref{prp:ICleICmin}. So the proof in case $h(x,R)=\vol{E\cap B_C(x,R)}\,(\vol{B_C(x,R))}^{-1}$ is completed.

For the remaining case, when $h(x,R)=\vol{B_C(x,R)}^{-1}\vol{B_C(x,R)\setminus E}$, using Lemma \ref{lem:I_C(v)> cv} and the fact that $I_C$ is non-decreasing by Proposition \ref{prp: I C is concave conical}, we argue as in Case~1 in Lemma~4.2 of \cite{le-ri} we get
\[
c_1/4\le (\eps\ell_2)^{1/(n+1)}
\]
This is a contradiction, since $\eps\ell_2<(c_1/4)^{n+1}$ by \eqref{eq:epsfine}.
\end{proof}

One of the consequences of Proposition~\ref{prp:leon rigot lem 42 coni} is the following lower density bound, which is usually obtained from the monotonicity formula.

\begin{corollary}[Lower density bound]
\label{cor:lwrdnbnd}
Let $C\subset \rr^{n+1}$ be a regular conically bounded convex body, and $E\subset C$ an isoperimetric region of volume $v$. Then there exists a constant $M>0$, only depending on the constant $\eps$ in \eqref{eq:epsfine}, on a Poincar\'e's constant for $r\le 1$ as in \eqref{eq:isnqgdbl1}, and on an Ahlfors constant $\ell_1$ as in \eqref{eq:isnqgdbl1a}, such that
\begin{equation}
\label{eq:lwrdnbnd}
\pp(E,B_C(x,r))\ge M r^n,
\end{equation}
for all $x\in\ptl_C E_1$ and $r\le 1$.
\end{corollary}
\begin{proof}
Let $E\subset C$ be an isoperimetric region of volume $v>1$, that exists by Proposition \ref{prp: exist isop reg conical}. The constant $\eps$ in \eqref{eq:epsfine} can be chosen independently of $v>1$ since the quantity $\inf_{v\ge 1}v^{-n/(n+1)}I_C(v)$ is uniformly bounded from below by a positive constant because of \eqref{eq:optimanondegen}. Then we have
\begin{equation*}
\begin{split}
\pp(E,B_C(x,r))&\ge M\min\{\vol{E\cap B_C(x,r)}, \vol{B_C(x,r) \setminus E}\}^{n/{(n+1)}}
\\
&=M\,(\vol{B_C(x,r)}\,h(x,r))^{n/(n+1)}\ge M(\vol{B_C(x,r)}\,\eps)^{n/(n+1)}
\\
&\ge M\,(\ell_1\eps)^{n/(n+1)}\,r^n,
\end{split}
\end{equation*}
as claimed.
\end{proof}


So we have our convergence result

\begin{theorem}
\label{thm:convrescnoi}
Let $C\subset\rr^{n+1}$ be a regular conically bounded convex body. Then a rescaling of a sequence of isoperimetric regions of volumes approaching infinity converges in Hausdorff distance to a geodesic ball centered at the vertex in the asymptotic cone. The same convergence result holds for their free boundaries.
\end{theorem}

\begin{proof}
Assume $0\in\ptl C$. Let $\{E_i\}_{i\in\nn}\subset C$ be a sequence of isoperimetric regions of volumes $\vol{E_i}\to\infty$, and let $\la_i\to 0$ so that $\vol{\la_iE_i}=1$. The sets $\Om_i=\la_iE_i$ are isoperimetric regions in $\la_iC$, and they are connected by Proposition~\ref{prp: I C is concave conical}. We claim
\begin{equation}
\label{eq:claim diamto00coni}
\diam(\Om _i)\le c,\qquad\text{for all}\ i\ \text{and some}\ c>0.
\end{equation}
If claim holds, let $q\in \intt(\clb_{C_{\infty}}(0,1))$ and $B_q \subset\intt(\clb_{\infty}(0,1))$ be a Euclidean geodesic ball. Consider a solid cone $K_q$ with vertex $q$ such that $0\in \intt(K_q)$ and $K_q\cap C\cap \ptl B(0,1)=\emptyset$.  By \eqref{eq:claim diamto00coni} we get $\diam(\la_i\Om_i)\to 0$, and hence $\la_i\Om_i\to 0$ in Hausdorff distance, what implies
\[
\la_i\Om_i\subset  K_q,
\]
for large enough $i\in \nn$.

As the sequence $\la_i^2C\cap\clb(0,1)$ converges in Hausdorff distance to $C_\infty\cap \clb(0,1)$, we construct using \cite[Thm.~3.4]{rv1} a family of bilipschitz maps
\[
f_i:\la_i^2 C\cap\clb(0,1)\to C_\infty\cap \clb(0,1)
\]
so that $f_i$ is the identity in $B_q$ and it is extended linearly along the segments leaving from $q$. The maps $f_i$ satisfy $\Lip(f_i),\,\Lip(f_i^{-1})\to 1$, and have the additional property
\[
\pp_{C_{\infty}}(f_i(\la_i\Om_i))=\pp_{\clb_{C_\infty}(0,1)}(f_i(\la_i\Om_i)).
\]
Then  $g_i=\la_if_i\la_i^{-1}$, defined from $\la_iC\cap \clb(0,\la_i^{-1})$ to $C_{\infty}\cap  \clb(0,\la_i^{-1})$ satisfy the same properties $\Lip(g_i)$, $\Lip(g_i^{-1})\to 1$ and $\pp_{C_{\infty}}(g_i(\Om_i))=\pp_{\clb_{C_\infty}(0,\la_i^{-1})}(g_i(\Om_i))$. From Lemma \ref{lem:bilip} we get
\begin{equation}
\label{eq:bigvolconi}
\begin{split}
\lim_{i\to \infty} \diam (\Om_i)&=\lim_{i\to \infty} \diam (g_i(\Om_i)),
\\
1=\lim_{i\to \infty} \vol{\Om_i}&=\lim_{i\to \infty} \vol{g_i(\Om_i)},
\\
\liminf_{i\to \infty} \pp_{\la_iC}(\Om_i)&= \liminf_{i\to \infty} \pp_{C_{\infty}}(g_i(\Om_i)).
\end{split}
\end{equation}
Consequently, by \eqref{eq:claim diamto00coni}, the sets $g_i(\Om_i)$ have uniformly bounded diameter. If the sequence of sets $\{g_i(\Om_i)\}_{i\in\nn}$ has a divergent subsequence, then \eqref{eq:optimanondegen7}, \eqref{eq:bigvolconi}, and Proposition \ref{prp: asympt ineq cone} imply
\begin{equation}
\label{eq:bigvolcon1}
I_{C_{\infty}}(1)=\lim_{i\to \infty}I_{\la_iC}(1)=\liminf_{i\to \infty} \pp_{C_{\infty}}(g_i(\Om_i))\ge I_H(1),
\end{equation}
and from \eqref{eq:isopkon} we would get that  $C_{\infty}$ is a half-space, a contradiction. Hence the sequence $\{g_i(\Om_i)\}_{i\in\nn}$ stays bounded, and we can apply the convergence results for convex bodies to obtain $L^1$-convergence of the sets $\Om_i$ and improve, using the density estimates in Proposition~\ref{prp:leon rigot lem 42 coni}, the $L^1$-convergence to Hausdorf convergence of the sets $\Om_i$ and their boundaries \cite[Theorems~5.11 and 5.13]{rv1}.  

So it only remains to prove \eqref{eq:claim diamto00coni} to conclude the proof.
Since $(\la_iC)_{\infty}=C_{\infty}$ we can choose, using  Lemma~\ref{lemma:isnqgdblconi}, a uniform Poincar\'e's constant for $r\le 1$, and a uniform Ahlfors constant $\ell_1$ for all $\la_iC$. Further, since $I_{\la_iC}\ge I_{C_{\infty}}$, the constant $\eps$ in \eqref{eq:epsfine} can be chosen uniformly for all $\la_iC$ as well. Consequently  a lower density bound, as in Corollary \ref{cor:lwrdnbnd}, holds for all $\Om_i$ with a uniform constant. Since the sets $\Om_i$ are connected by Proposition~\ref{prp: I C is concave conical}, we conclude that $\diam(\Om_i)$ are uniformly bounded, since otherwise \eqref{eq:lwrdnbnd} would imply that $\pp_{\la_i C}(\la_i E_i)$ goes to infinity. This way we obtain a contradiction, since by \eqref{eq:half-plane}, we get $\pp_{\la_i C}(\la_i E_i)=I_{\la_i C}(1)\le I_{H}(1)$ for all $i$. 
\end{proof}

Since we are assuming smoothness of the boundaries of both the conically bounded set $C$ and of its asymptotic cone $C_\infty$ (out of the vertex), we can use density estimates for varifolds to improve the convergence. In particular, the mean curvatures of the boundaries of the isoperimetric regions satisfy a uniform estimate

\begin{lemma}
\label{lem:boundmeancurv}
Let $C\subset\rr^{n+1}$ be a regular conically bounded convex body, and $\{E_i\}_{i\in\nn}$ a sequence of isoperimetric regions of volumes $v_i\to\infty$. Let $H_i$ be the constant mean curvature of the regular part of the boundary of $E_i$. Then $H_iv_i^{1/(n+1)}$ is bounded.
\end{lemma}

\begin{proof}
It is known that the mean curvature $H$ of the boundary of an isoperimetric region of volume $v$ satisfies $H\le (I_C')_-(v)$, where $(I_C')_-$ is the left derivative of the concave function $I_C$. Observe that there are constants $m$, $M>0$ such that
\[
m v^{n/(n+1)}\le I_C(v)\le M v^{n/(n+1)}, \qquad \text{for large}\ v.
\]
The left inequality follows from inequality \eqref{eq:icgeicinfty}, $I_C\ge I_{C_\infty}$, and it is indeed true for any $v>0$. The second one follows from \eqref{eq:optimanondegen}, $\lim_{v\to\infty} (I_{C_\infty}^{-1}I_C)(v)=1$.

For large $v$ we have
\[
v^{1/(n+1)}H\le \bigg(\frac{1}{m}\bigg)^{1/n}I_C(v)^{1/n}(I_C')_-(v)=
\bigg(\frac{1}{m}\bigg)^{1/n}\bigg(\frac{n}{n+1}\bigg)\big(Y_C\big)'_-(v),
\]
where $Y_C=I_C^{(n+1)/n}$. Hence the estimate
\[
(Y_C)'_-(v)=\lim_{h\to 0^+}\frac{Y_C(v-h)-Y_C(v)}{h}\le \frac{Y_C(v)}{v}\le M^{(n+1)/n}.
\]
proves the result.
\end{proof}

\section{Large isoperimetric regions in conically bounded convex bodies of revolution}
\label{sec:revolution}

In this Section we consider regular conically bounded sets of revolution in $\rr^{n+1}$,  generated by a smooth convex function $f:[0,+\infty)\to\rr^+$ with $f(0)=f'(0)=0$. We may think of $f$ as the restriction to $[0,+\infty)$ of a smooth convex function $f:\rr\to\rr^+$ satisfying $f(x)=f(-x)$. For any $n\in\nn$, the function $f$ defines a convex body of revolution $C_f\subset\rr^{n+1}$ as the set of points $(x,y)\in\rr^n\times \rr$ satisfying the inequality $y\ge f(|x|)$. As we shall see, the conical boundedness condition is equivalent to the existence of a constant $a>0$ so that
\[
\lim_{x\to\infty}(f(x)-ax)=0.
\]
This implies that the line $y=ax$ is an asymptote of the function $f$. For such a function, we have
\[
\lim_{x\to\infty}\frac{f(x)}{x}=a.
\]
and L'H\^opital's Rule implies
\[
\lim_{x\to\infty} f'(x)=\lim_{x\to\infty}\frac{f(x)}{x}=a,
\]
and
\[
\lim_{x\to\infty} xf''(x)=\lim_{x\to\infty}\frac{f'(x)}{\log(x)}=0.
\]

We have the following

\begin{lemma}
\label{lem:foliation}
Given a smooth convex function $f:[0,+\infty)\to\rr^+$ such that $f'(0)=0$ and $\lim_{x\to+\infty}(f(x)-ax)=0$ for some constant $a>0$, we have
\begin{enumerate}
\item[(i)] The set $C_f=\{x,y)\in\rr^n\times\rr:y\ge f(|x|)\}$ is conically bounded with asymptotic cone at infinity $(C_f)_\infty=\{(x,y)\in\rr^n\times\rr: y\ge a|x|\}$.
\item[(ii)] There exists a compact set $K\subset C_f$ so that $C_f\setminus K$ is foliated by spherical caps meeting $\ptl C_f$ in an orthogonal way.
\item[(iii)] The mean curvature of the spherical caps is a non-increasing function $($in the unbounded direction$)$ and converges to $0$.
\end{enumerate}
\end{lemma}

\begin{proof}
Let us call $C^\infty=\{(x,y)\in\rr^n\times\rr : y\ge a |x|\}$. Observe that Lemma~\ref{lem:convexasymp} implies that $f(x)\ge a x $ for all $x\ge 0$ and so $C_f\subset C^\infty$. To show that the set $C_f$ is conically bounded we compute $\rho((C_f)_{f(x)},u)=x$, and $\rho((C^\infty)_{f(x)},u)=f(x)/a$  for all $u\in\esf^{n-1}$. Hence condition \eqref{eq:def ex asymp con} is satisfied. We know that the asymptotic cone $(C_f)_\infty$ is the epigraph of the convex function $f_\infty(x)=\lim_{\mu\to\infty} \mu^{-1}f(\mu x)=ax$. This implies (i).

Let us prove (ii). For any $x>0$, we consider the center $(0,c(x))$ and the radius $r(x)$ of the circle meeting the graph of $f$ orthogonally at the point $(x,f(x))$. We have
\[
c(x)=f(x)-x\,f'(x), \qquad r(x)=x\,(1+f'(x)^2)^{1/2}.
\]
It is easy to check that $c'(x)=-x\,f''(x)\le 0$. If we define $g(x)=c(x)+r(x)$ and fix $x_0>0$, the circles around the one with center $(0,c(x_0))$ and radius $r(x_0)$ form a local foliation if $g'(x_0)>0$. Since
\begin{equation*}
g'(x)=x\,f''(x)\,\bigg(-1+\frac{f'(x)}{(1+f'(x)^2)^{1/2}}\bigg)+(1+f'(x)^2)^{1/2},
\end{equation*}
taking limits we obtain 
\[
\lim_{x\to\infty} g'(x)=(1+a^2)^{1/2}>0.
\]
So we conclude that there exists $x_m>0$ so that the circles corresponding to points $x>x_m$ form a foliation meeting the boundary of the convex set in an orthogonal way. The corresponding bodies of revolution exhibit the same property. In these cases, there is a foliation outside a compact set whose leaves are spherical caps meeting orthogonally the boundary of the convex set.

To prove (iii), simply take into account that the mean curvature of the spheres is $r(x)^{-1}=x^{-1}(1+f'(x)^2)^{-1/2}$ and $\lim_{x\to\infty} r(x)^{-1}=0$.
\end{proof}

\begin{remark}
\label{rem:half-space}
Let $C$ be a convex body of revolution generated by a convex function $f$ satisfying $f'(0)=0$. If we assume $\lim_{x\to\infty} x^{-1}f(x)=0$ then $f\equiv 0$. This follows since the function $f'$ is non-decreasing and satisfies $\lim_{x\to\infty} f'(x)=0$. Hence a convex body of revolution cannot be asymptotic to a half-space unless it is a half-space.
\end{remark}

Let $(M,g_0)$ be a smooth Riemannian manifold with smooth boundary. Assume that $\Sg$ is an embedded hypersurface with constant mean curvature $H_\Sg$ and that $\ptl\Sg$ is contained in $\ptl M$ and meets $\ptl M$ in an orthogonal way. We shall assume that $\Sg$ is two-sided and so there is a unit normal $N_\Sg$ to $\Sg$. The unit conormal to $\ptl\Sg$ will be denoted by $\nu_\Sg$.

Let $X$ be a $C^\infty$ complete vector field in $M$ so that $X|_{\Sg}=N$ and $X|_{\ptl M}$ is tangent to $\ptl M$. The flow $\{\varphi_t\}_{t\in\rr}$ of $X$ preserves the boundary of $M$ and allows us to define ``graphs'' over $\Sg$. If $u\in C^{2,\alpha}(\Sg)$ has small enough $C^{2,\alpha}$ norm, then the graph of $u$, denoted by $\Sg(u)$, is defined as the set $\{\varphi_{u(p)}(p):p\in\Sg\}$. For small $C^{2,\alpha}$ norm, $\Sg(u)$ is an embedded hypersurface. Given a Riemannian metric $g$ on $M$, we shall denote the unit normal to $\Sg(u)$ in $(M,g)$ by $N^g_{\Sg(u)}$ and shall drop $g$ when $g=g_0$. The unit conormal will be denoted by $\nu_{\Sg(u)}^g$. Given $g$, the inner unit normal to the boundary of $M$ will be denoted by $N_{\ptl M}^g$. The laplacian on $\Sg$, the Ricci curvature tensor, the second fundamental form of $\ptl M$ with respect to an inner normal, and the squared norm of the second fundamental form, with respect to a Riemannian metric $g$, will be denoted by $\Delta_\Sg^g$, $\Ric^g$, $\II^g$, $|\sg^g|^2$, respectively. We shall drop the superscript $g$ when $g=g_0$.

We shall use the following well-known result, compare with \cite[Prop.~10]{ambrozio}

\begin{proposition}
\label{prop:ift}
Let $(M,g_0)$ be a Riemannian manifold with smooth boundary and $\Sg\subset M$ an embedded hypersurface with constant mean curvature $H_\Sg$ such that $\ptl\Sg\subset\ptl M$ meets $\ptl M$ in an orthogonal way. Assume that the free boundary problem
\begin{equation}
\label{eq:freeboundary}
\begin{split}
\Delta_\Sg u+(\Ric(N,N)+|\sg|^2)\,u&=0, \qquad \text{on}\ \Sg
\\
\frac{\ptl u}{\ptl \nu_\Sg}+\II(N,N)\,u&=0, \qquad \text{on}\ \ptl \Sg
\end{split}
\end{equation}
has just the trivial solution. Then there is a neighborhood $U$ of $g_0$ in $\text{Riem}(M)$ and a neighborhood $I$ of $H_\Sg$ so that for $(g,H)\in U\times I$, there is just one graph of class $C^{2,\alpha}$ with constant mean curvature $H$ meeting $\ptl M$ in an orthogonal way in the Riemannian manifold $(M,g)$.
\end{proposition}

\begin{proof}
The proof is an application of the Implicit Function Theorem in Banach spaces. Consider the map $\Phi:(\text{Riem}(M)\times\rr)\times C^{2,\alpha}(\Sg)\longrightarrow C^{0,\alpha}(\Sg)\times C^{1,\alpha}(\ptl\Sg)$ defined by
\[
\Phi(g,H,u)=(H^g_{\Sg(u)}-H_\Sg,g(\nu^g_{\Sg(u)},N_{\ptl M}^g)).
\]
The partial derivative $D_2\Phi$ with respect to the factor $C^{2,\alpha}(\Sg)$ is given by
\[
-D_2\Phi(g_0,H_0,0)(v)=\big(\Delta_\Sg v+(\text{Ric}(N,N)+|\sg|^2)\,v, \frac{\ptl v}{\ptl \nu_\Sg}+\text{II}(N_\Sg,N_\Sg)\,v\big).
\]
This map is injective by assumption and surjective by the Fredholm alternative. It is continuous and an isomorphism by Schauder estimates \cite[Theorem~6.30 (6.77)]{MR1814364}. Hence we can apply the Implicit Function Theorem for Banach spaces to conclude the proof.
\end{proof}

We shall also need the following

\begin{lemma}[{\cite[Corollary~3.4]{MR1356151}}]
\label{lem:neumann}
Let $\esf^n(R)\subset\rr^{n+1}$ and $B(r)\subset\esf(R)$ be a geodesic ball (spherical cap) of radius $0<r<\pi R/2$. Then the first nonzero Neumann eigenvalue $\mu(r)$ of the Laplacian in $B(r)$ satisfies $\mu(r)>nR^{-2}$.
\end{lemma}

Now we are in position to prove the main result in this Section

\begin{theorem}
\label{thm:foliation}
Let $C$ be a conically bounded convex body of revolution. Then there exists $v_0>0$ such that any isoperimetric region $E\subset C$ of volume $|E|\ge v_0$ is a spherical cap meeting the boundary of $C$ in an orthogonal way.
\end{theorem}

\begin{proof}
By Remark~\ref{rem:half-space}, the asymptotic cone of $C$ is not a half-space. Hence $C$ is generated by a convex function $f$ such that $\lim_{x\to\infty} x^{-1}f(x)=a>0$. The asymptotic cone of $C$ is $C_\infty$ is the convex body of revolution generated by the function $f_\infty(x)=ax$.

Let $\{E_i\}_{i\in\nn}$ be a sequence of isoperimetric regions in $C$ with $|E_i|\to\infty$. By Theorem~\ref{thm:convrescnoi}, for $\la_i=v_i^{-1/(n+1)}$, the boundaries of $\la_iE_i$ converge in Hausdorff distance to a spherical cap $\Sg\subset\esf(R)$, of radius $0<r<\pi R/2$, inside the asymptotic cone of $C$. Moreover, we can find a sequence of diffeomorphisms $\varphi_i$ of class $C^\infty$ applying a small tubular neighborhood of $\Sg$ into a subset of $\la_iC$ containing the boundary of $\la_iE_i$. The mean curvature of the boundary of $\la_iE_i$ is given by $H_iv_i^{1/(n+1)}$, which is uniformly bounded by Lemma~\ref{lem:boundmeancurv}, and so it is the mean curvature of $\varphi^{-1}_i(\la_iE_i)$ computed with respect to the metric $\varphi^{-1}(g_0)$. The reduced boundary of $\varphi_i(\la_iE_i)$ is a stationary varifold with boundary. Since the perimeters of $\varphi_i(\la_iE_i)$ converge to the perimeter of $\Sg$, we can use \cite[Theorem~4.13]{MR863638} to get $C^{1,\delta}$-convergence of the boundaries. By elliptic regularity, the mean curvatures of the boundaries of $\varphi^{-1}_i(\la_iE_i)$, computed with respect to the metric $\varphi^*g_0$, also converge to the mean curvature of $\Sg$, and the boundary of $\varphi_i(\la_iE_i)$ is the graph of a $C^\infty$ function over $\Sg$ in the sense defined above.

The hypersurface $\Sg\subset C_\infty$ is the boundary of an isoperimetric region in $C_\infty$. On $\Sg$ we have $\Ric(N,N)+|\sg|^2=nR^{-2}$ and $\II(N,N)=0$. So the free boundary problem \eqref{eq:freeboundary} is given by
\begin{align*}
\Delta u+nR^{-2}u&=0, \qquad \text{on}\ \Sg,
\\
\frac{\ptl u}{\ptl \nu}&=0,\qquad \text{on}\ \ptl\Sg.
\end{align*}
By Lemma~\ref{lem:neumann} the first nonzero Neumann eigenvalue of the Laplacian on $\Sg$ is strictly larger than $nR^{-2}$, and so the only solution is $u=0$. Proposition~\ref{prop:ift} then implies that, for large enough $i\in\nn$ so that $\varphi^*g_0$ is close to $g_0$ and the mean curvature of the boundary of $\la_iE_i$ is close to the one of $\Sg$, there is only one such graph.

Consider now a sequence of spherical caps in $C$ with the same mean curvature as the one of $\ptl E_i$. Scaling down we have $C^\infty$ convergence to $\Sg$. By the uniqueness part of Proposition~\ref{prop:ift}, we obtain that $E_i$ is a spherical cap for $i$ large enough.
\end{proof}

\bibliography{convex}
\end{document}